\documentclass[12pt,a4paper]{article}
\usepackage{latexsym,amssymb,amsmath,amsthm}
\usepackage{enumitem}
\usepackage{xcolor}
\usepackage{graphicx}
\usepackage{subfig}

\newtheorem{theorem}{Theorem}[section]
\newtheorem*{theorem*}{Theorem}
\newtheorem{lemma}[theorem]{Lemma}
\newtheorem{corollary}[theorem]{Corollary}

\newtheorem{problem}[theorem]{Problem}

\newcommand{\claimproofend}{\hspace*{.1mm}\hspace{\fill}}

\theoremstyle{remark}
\newtheorem{remark}[theorem]{Remark}

\newcommand{\size}[1]{\left|#1\right|}

\newcommand\Setx[1] {\left\{{#1}\right\}}

\newcommand{\CC}{\mathcal C}
\newcommand{\RP}[1]{\mathbb{R}\mathrm{P}^{#1}}

\newcommand{\ZZ}{\mathbb Z}

\graphicspath{{./figures/}}

\begin{document}
\title{\textbf{Colouring normal quadrangulations of projective
    spaces}}

\author{
Tom\'a\v s Kaiser\footnote{Department of Mathematics \& NTIS, 
University of West Bohemia in Pilsen, 
Czech Republic}
  \and On-Hei Solomon Lo\footnote{School of Mathematical Sciences,
Tongji University,
Shanghai  200092, China}
  \and Atsuhiro Nakamoto\footnote{Faculty of Environment and Information Sciences,
Yokohama National University,
Yokohama 240-8501, Japan}
  \and Yuta Nozaki$^{\ddag}$\footnote{WPI-SKCM$^2$, Hiroshima University, Higashi-Hiroshima 739-8526, Japan}
  \and Kenta Ozeki$^{\ddag}$
}

\date{}

\maketitle
\begin{abstract}
  Youngs proved that every non-bipartite quadrangulation of the
  projective plane $\RP 2$ is 4-chromatic. Kaiser and Stehl\'{\i}k
  [J.\ Combin.\ Theory Ser.\ B 113 (2015), 1--17] generalised the notion
  of a quadrangulation to higher dimensions and extended Youngs'
  theorem by proving that every non-bipartite quadrangulation of the
  $d$-dimensional projective space $\RP d$ with $d \geq 2$ has
  chromatic number at least $d+2$.  On the other hand, Hachimori et
  al.\ [European.\ J.\ Combin.\ 125 (2025), 104089] defined another kind
  of high-dimensional quadrangulation, called a {\em normal
    quadrangulation}.  They proved that if a non-bipartite normal
  quadrangulation $G$ of $\RP d$ with any $d \geq 2$ satisfies a
  certain geometric condition, then $G$ is $4$-chromatic, and asked
  whether the geometric condition can be removed from the result. In
  this paper, we give a negative solution to their problem for the
  case $d=3$, proving that there exist 3-dimensional normal
  quadrangulations of $\RP 3$ whose chromatic number is arbitrarily
  large. Moreover, we prove that no normal quadrangulation of $\RP d$
  with any $d \geq 2$ has chromatic number $3$.
\end{abstract}

\section{Introduction}
\label{sec:introduction}

Let $G$ be a graph,
and let $c\colon V(G) \to \{1,\ldots,k\}$ be a map.
We say that $c$ is a {\em $k$-colouring\/} if every edge $xy \in E(G)$ satisfies $c(x) \ne c(y)$.
A graph $G$ is said to be {\em $k$-colourable\/} if $G$ admits a $k$-colouring.
The {\em chromatic number\/} of $G$, denoted by $\chi(G)$,
is the smallest $k$ such that $G$ is $k$-colourable,
and if $\chi(G)=k$, then $G$ is said to be {\em $k$-chromatic}.
A {\em $k$-cycle\/} $C$ means a cycle of length $k$,
and we say that $C$ is {\em even\/} (respectively, {\em odd\/})
if $k$ is even (respectively, odd).

The {\em Four Colour Theorem\/} \cite{4ct1,4ct2}, one of the most
celebrated results in graph theory, states that every planar graph is
$4$-colourable.  This result has been extended in many directions, yet
in this paper we consider colourings of graphs embedded on
non-spherical closed surfaces.
We consider only {\em 2-cell embeddings} of graphs on surfaces, meaning embeddings whose faces are homeomorphic to an open 2-cell.

Heawood~\cite{Heawood} proved that every graph $G$ embedded in a
closed non-spherical surface $S$ satisfies
$$\chi(G) \leq \frac {7+\sqrt{24k+1}}2,$$
where $k$ denotes the {\em Euler genus\/} of $S$, that is,
$k=2-\epsilon(S)$ with $\epsilon(S)$ denoting the {\em Euler
  characteristic\/} of $S$.  (This bound is known to be best possible
except for the Klein bottle \cite{Ringel}.)  Hence, if we let $G$ be a
graph embeddable in a closed surface $S$ with Euler genus $k>0$, then
we have $\chi(G)=O(\sqrt{k})$.  On the other hand,
Thomassen~\cite{Thoma} proved that
for any closed surface $S$ with Euler genus $k>0$, there exists a
positive integer $N(S)$ such that if $G$ is a graph embedded in $S$
with edge-width at least $N(S)$, then $G$ is $5$-colourable. (Recall
that {\em edge-width\/} of a graph $G$ on a surface $S$ is defined to
be the length of a shortest non-contractible cycle of $G$.) Note that
the value $5$ is best possible, since every non-spherical surface $S$
admits non-$4$-colourable graphs of arbitrary edge-width. In a more informal way, Thomassen’s result can be rephrased as stating that \emph{locally planar} graphs on any non-spherical surface (that is, embedded graphs without short non-contractible cycles) are 5-colourable.

Now we turn our attention to \emph{even-sided} graphs on a closed
surface $S$, that is, simple graphs embedded in $S$ such that each face is bounded
by an even cycle. In particular, a \emph{quadrangulation} of $S$ is an
even-sided graph on $S$ with each face being quadrangular.

It is easy to see that every even-sided graph on the sphere is bipartite, but
this does not hold for any non-spherical closed surface.  Hutchinson
\cite{hutch2} proved that if $G$ is an even-sided graph on $S$ with
Euler genus $k>0$, then
\[
\chi(G) \leq \frac {5+\sqrt{16k-7}}2.
\]
(Based on recent developments \cite{Liu} in this topic, it is known
that this bound is best possible, except for the Klein bottle and the
double torus.)  Regarding the locally planar version, Hutchinson
\cite{hutch} proved that every locally planar even-sided graph on any
orientable surface is 3-colorable, with the constant $3$ being best
possible. However, this result does not hold even for the projective
plane $\RP 2$, as demonstrated by Youngs in the following theorem
\cite{young}.

\begin{theorem} [Youngs, 1996] \label{t-young}
Every non-bipartite quadrangulation of the projective plane $\RP 2$
is $4$-chromatic.
\end{theorem}

This result implies that the projective plane admits non-3-colourable
even-sided graphs with arbitrarily large edge-width, highlighting a
significant difference between phenomena on orientable and
nonorientable surfaces. Therefore, Theorem~\ref{t-young} prompts
further in-depth studies on the colouring of even-sided graphs on
surfaces, as discussed in \cite{ahnno, ms, nno}.

In this paper, we consider higher-dimensional extensions of the notion
of quadrangulation, particularly in relation to their chromatic number.
One such extension was introduced by the first author and Stehl\'{\i}k~\cite{kais}, as follows.

Let ${\cal K}$ be a $d$-dimensional {\em generalized\/} simplicial
complex, where $d \geq 2$ (in this generalized version, two simplices
may share the same vertex set).  A {\em KS-quadrangulation\/} $G$ of
${\cal K}$ is a spanning subgraph of the 1-skeleton ${\cal K}^{(1)}$
of ${\cal K}$ such that for every $d$-dimensional simplex $t$ of
${\cal K}$, the intersection of $G$ and $t$ is a complete bipartite
graph of $d+1$ vertices with at least one edge.  (See Figure
\ref{3dquad}.) When the underlying space of ${\cal K}$ is homeomorphic
to a topological space $X$, we say $G$ is a {\em
  KS-quadrangulation\/} of $X$.  It is not hard to see that a
2-dimensional KS-quadrangulation coincides with an ordinary
quadrangulation of $X$.

\begin{figure}[htb]
\begin{center}
\includegraphics[width=130mm]{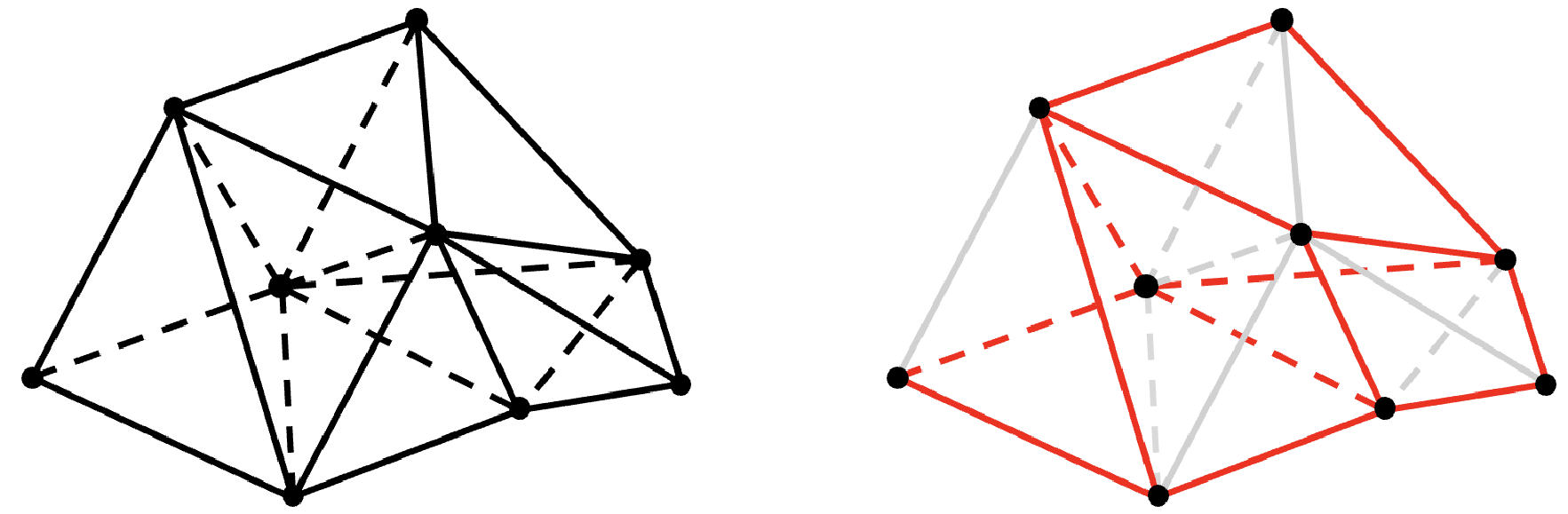}
\end{center}
\caption{The local structure of a $3$-dimensional
  KS-quadrangulation. Left: A portion of the underlying simplicial
  complex. Right: A part of the KS-quadrangulation (in
  red). 
}
\label{3dquad}
\end{figure}

The following theorem, proved in~\cite{kais} for KS-quadrangulations
of the $d$-dimensional projective space $\RP d$, extends
Theorem~\ref{t-young}.

\begin{theorem} [Kaiser and Stehl\'{\i}k, 2015] \label{ks}
 Let $d \geq 2$, and let $G$ be a non-bipartite KS-quadrangulation of
 the $d$-dimensional projective space $\RP d$.
 Then the chromatic number of $G$ is at least $d+2$.
\end{theorem}

It is also shown in~\cite{kais} that for any $d \geq 3$ and
$t \geq 5$, the complete graph $K_t$ on $t$ vertices can be embedded
in $\RP d$ as a KS-quadrangulation provided that $t-d$ is even. Thus,
there is no constant upper bound on the chromatic number of
KS-quadrangulations of $\RP d$ for $d \geq 3$.

On the other hand, Hachimori et al.\ \cite{hachizono} defined a {\em
  normal quadrangulation} of a $d$-dimensional manifold $X$ as
the $1$-skeleton of a {\em cubical complex\/} ${\cal C}$ whose
underlying space is homeomorphic to $X$, that is, one each of
whose face is either a point, a segment, a quadrilateral, or more
generally a cube of any dimension.  As in the case of
KS-quadrangulations, for $d=2$, a normal quadrangulation $G$ of
$X$ coincides with an ordinary quadrangulation of $X$,
since $G$ is a graph embedded in $X$ such that every
$2$-dimensional face is a quadrilateral.

Hachimori et al.~\cite{hachizono} studied a particular type of
normal quadrangulations of the $d$-dimensional projective space
$\RP d$, known as \emph{zonotopal quadrangulations} (we refer to their
paper for the necessary definitions). In dimension $2$, zonotopal
quadrangulations of $\RP 2$ can be obtained from rhombic tilings of
regular polygons (cf.~\cite{hama}).

The main result of~\cite{hachizono} is the following.
\begin{theorem}
[Hachimori, Nakamoto, and Ozeki, 2025]
\label{Hachimori et al., 2024+}
Let $d \geq 2$ be any integer.
Then every non-bipartite zonotopal quadrangulation
of the $d$-dimensional projective space $\RP d$ is $4$-chromatic.
\end{theorem}

This led the authors to ask if this phenomenon extended to general
normal quadrangulations of projective spaces.

\begin{problem} \label{P:hachi}
Let $d \geq 2$ be any integer.
Is every non-bipartite normal quadrangulation of the $d$-dimensional projective space $\RP d$
$4$-chromatic?
\end{problem}

In this paper, we give a negative solution to Problem~\ref{P:hachi} by proving the following theorem for $d=3$.

\begin{theorem} \label{t:complete}
  For any $n\geq 1$, there is a normal quadrangulation of $\RP 3$
  whose $1$-skeleton contains the complete graph $K_n$.
\end{theorem}

We do not know whether Problem~\ref{P:hachi} holds for the case $d \geq 4$.
We leave this as a problem:

\begin{problem} \label{P:new}
Let $d \geq 4$ and $t \geq 4$ be any integers.
Is there a non-bipartite normal quadrangulation $G$ of the $d$-dimensional projective space $\RP d$
such that $G$ is not $t$-colourable?
\end{problem}

On the other hand, we prove that Theorem~\ref{t-young} extends at least partially to higher dimensions, in the following sense.

\begin{theorem}\label{t:3col}
  No non-bipartite normal quadrangulation of $\RP d$ for any $d\geq 2$ is
  $3$-colourable.
\end{theorem}

The proof employs a novel application of intersection theory in homology, which is equivalently expressed in terms of the ring structure of cohomology groups.

This paper is organized as follows.  In Section~\ref{s:complete}, we
prove Theorem~\ref{t:complete}.  In Section~\ref{sec:RP2}, we provide
an alternative proof of Theorem~\ref{t-young}; extending this approach
to higher dimensions, we prove Theorem~\ref{t:3col} in Section~\ref{sec:RPd},
which follows as a corollary (Corollary~\ref{cor:3col_of_RPd}) of a more general result (Theorem~\ref{thm:3col_of_X}).

\section{Complete graphs in projective normal quadrangulations}
\label{s:complete}

In this section, we will show that normal quadrangulations of $\RP 3$
may contain arbitrarily large complete subgraphs, proving
Theorem~\ref{t:complete}. To this end, we will use the folloing corollary of a theorem established in~\cite{Erickson}.

\begin{theorem} [Erickson, 2014] \label{t:jeff} Let $S$ be a closed
  surface in $\mathbb{R}^3$ and let $Q$ be a bipartite quadrangulation of $S$. 
  Then $Q$ extends to a normal quadrangulation of the interior domain of $S$.
\end{theorem}

Let us proceed to the proof of the main result of this section.

\begin{proof}[Proof of Theorem~\ref{t:complete}]
Let $n=2k+1$ be an odd integer at least 3,
and we construct a normal quadrangulation of $\RP 3$
containing the complete graph with $n+1$ vertices
in the following argument.

Consider the 3-dimensional space with $x$-, $y$-, and $z$-axes.  On
the $xz$-plane, denoted by $P_{y=0}$, we let $R$ be the triangular
region bounded by the $z$-axis and two straight lines $z=-x+2$ and
$z=x-2$.  (See Figure~\ref{xyz} (left).)  Let $B_0$
denote the 3-manifold obtained by rotating $R$ around the $z$-axis.
Note that $B_0$ is homeomorphic to a $3$-dimensional ball.  Let
$\partial B_0$ denote the boundary of $B_0$, which is homeomorphic to
a 2-dimensional sphere. We first construct a quadrangulation of
$\partial B_0$ as follows:

\begin{figure}[htb]
\begin{center}
\includegraphics[width=110mm]{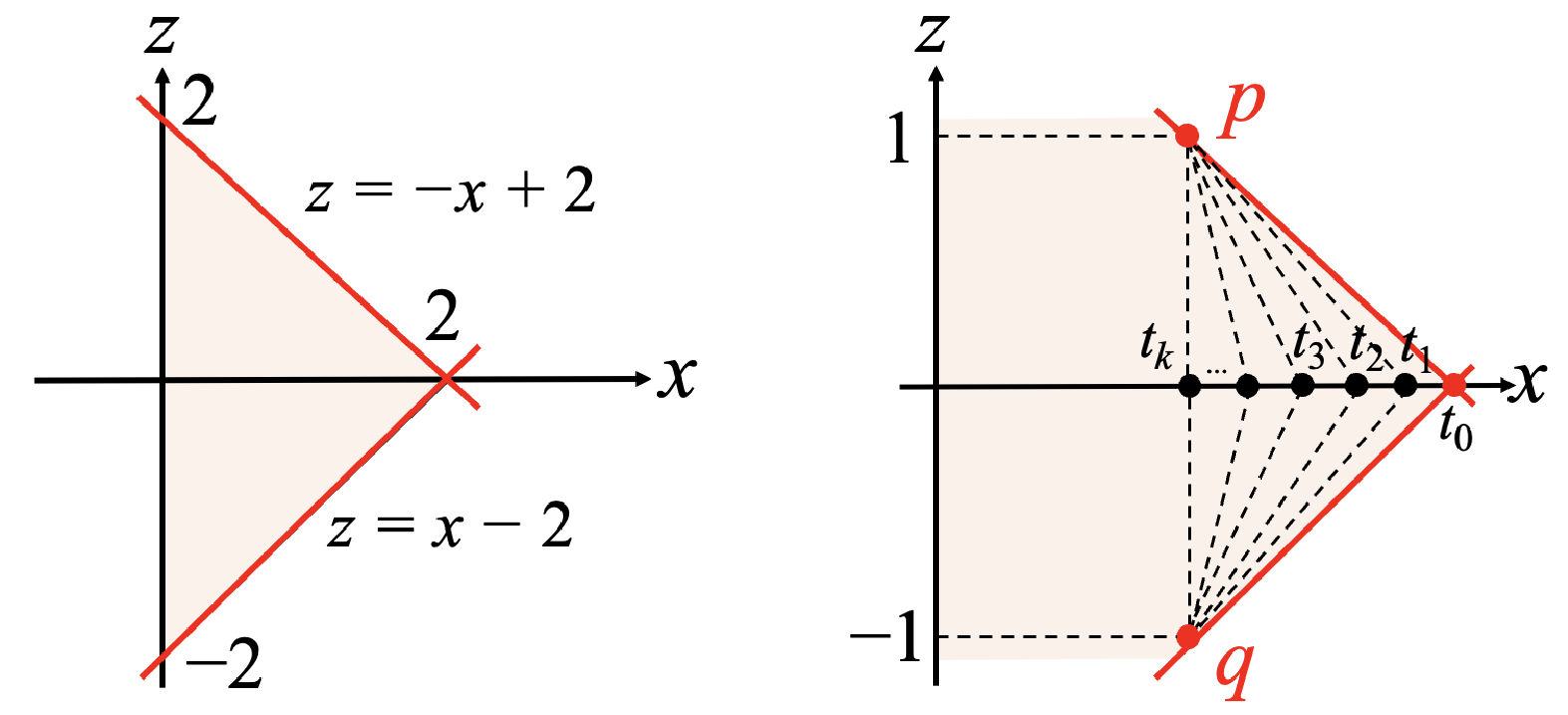}
\end{center}
\caption{Left: A triangular region $R$ in the proof of
  Theorem~\ref{t:complete}. Right: and segments $\boldsymbol{p}\boldsymbol{t}_i$
  and $\boldsymbol{q}\boldsymbol{t}_i$, $i=0,\dots,k$.}
\label{xyz}
\end{figure}

Let $\boldsymbol{p}=(1,0,1)$ and $\boldsymbol{q}=(1,0,-1)$,
and let $P$ and $Q$ be the orbits of $\boldsymbol{p}$ and $\boldsymbol{q}$ in $B_0$,
respectively, by the rotation of $R$ around the $z$-axis.
Put distinct $2n$ vertices $v_1,u_1, \ldots, v_{n}, u_{n}$ along $P$ in this order,
and $2n$ vertices $-v_1,-u_1, \ldots, -v_{n}, -u_{n}$ along $Q$ so that
$v_j$ and $-v_j$, and $u_j$ and $-u_j$
are located in the antipodal position on $\partial B_0$,
for all $j=1, \ldots, n$.
We regard $P$ and $Q$ as $2n$-cycles of the graph we construct,
and colour each $v_i$ and $-u_i$ black and each $u_i$ and $-v_i$ white.
Put a vertex $v_0$ and $-v_0$ on the points $(0,0,2)$ and $(0,0,-2)$, respectively.
Join $v_0$ to $v_1,\ldots,v_{n}$, and $-v_0$ to $-v_1,\ldots,-v_{n}$ on $\partial B_0$.
(See Figure~\ref{t0} (left).)

\begin{figure}[htb]
\begin{center}
\includegraphics[width=130mm]{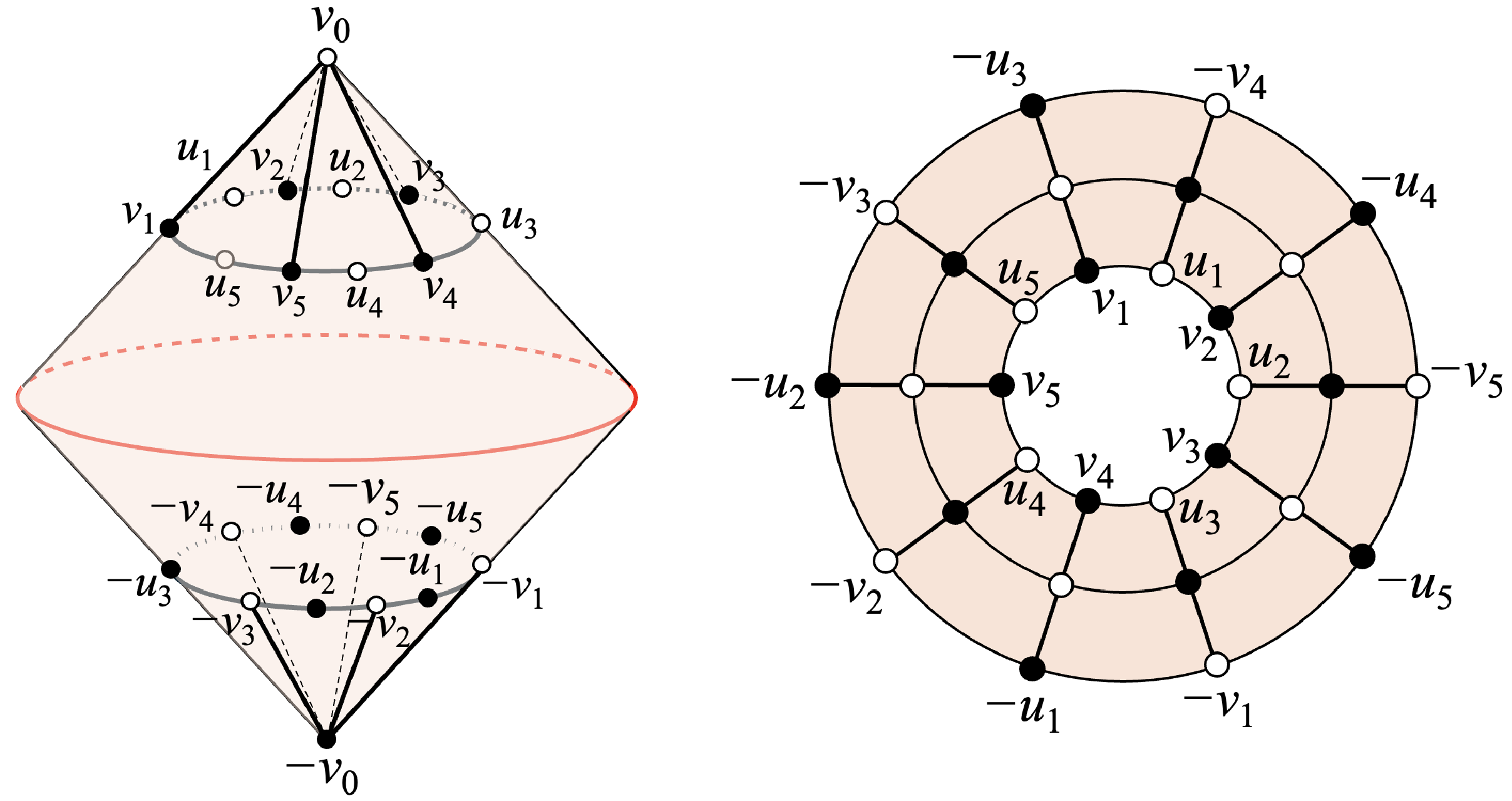}
\end{center}
\caption{Left: The ball $B_0$ in the proof of
  Theorem~\ref{t:complete} with $k=2$. Right: The even-sided graph $H$.}
\label{t0}
\end{figure}

Next we quadrangulate $A_0$ on $\partial B_0$, where $A_0$ is the
annulus on $\partial B_0$ bounded by $P$ and $Q$.  Let $H$ be the
Cartesian product of a cycle of length $2n$ and a path of length $2$
embedded on the annulus $A_0$, as shown in Figure~\ref{t0} (right). We
require that the subgraph of $H$ induced by the boundary of $A_0$ is
precisely the cycles $P$ and $Q$, and $v_1$ and $-u_{k+1}$ are joined
by a path of length $2$, so that the embedding of $H$ is symmetric
about the origin. It is readily seen that $H$ is a bipartite graph
with an even number of faces. Moreover, by the construction of the
graph, we have a quadrangulation, denoted by $G$, of $\partial B_0$
which has a point-symmetry $\rho$ with respect to the origin
$(0,0,0)$.

Let $\boldsymbol{t}_i$ be the $k+1$ points of $R$ with the coordinates
$$\boldsymbol{t}_i= \left( 2-\frac i{k}, 0,0 \right),$$
for $i=0,1,\ldots,k$.
(See Figure~\ref{xyz} (right).)
In the interior domain in $B_0$,
we take the line segments $\boldsymbol{pt}_i$ and $\boldsymbol{qt}_i$ in $R$ on $P_{y=0}$,
for $i=1,\ldots,k$, and we let $A_i$ be the annulus obtained
by rotating $\boldsymbol{pt}_i \cup \boldsymbol{qt}_i$ around the $z$-axis.
(Recall that $A_0$ can be viewed as the annulus obtained by this construction with $\boldsymbol{pt}_0 \cup \boldsymbol{qt}_0$.)
Then, for each $i=0,1,\ldots,k$, $A_i$ has $P$ and $Q$ as its boundary components.
For each $i=1,\ldots,k$,
we join $v_j$ to $-v_{j+i}$ and join $u_j$ to $-u_{j+i}$ with an edge on $A_i$, respectively,
for $j=1,\ldots,n$,
where the subscripts are taken modulo $n$.
See Figure~\ref{t1t2}.
If we let $G_i$ be the resulting graph on $A_i$,
then $G_i$ is a bipartite quadrangulation of $A_i$
with even number of faces, for $i=1,\ldots,k$.
In the graph $G_1 \cup \cdots \cup G_k$, 
we observe that for $j=1,\ldots, n=2k+1$, $v_j$ is adjacent to $-v_{j+1}, -v_{j+2}, \ldots, -v_{j+k}$.
 
\begin{figure}[htb]
\begin{center}
\includegraphics[width=130mm]{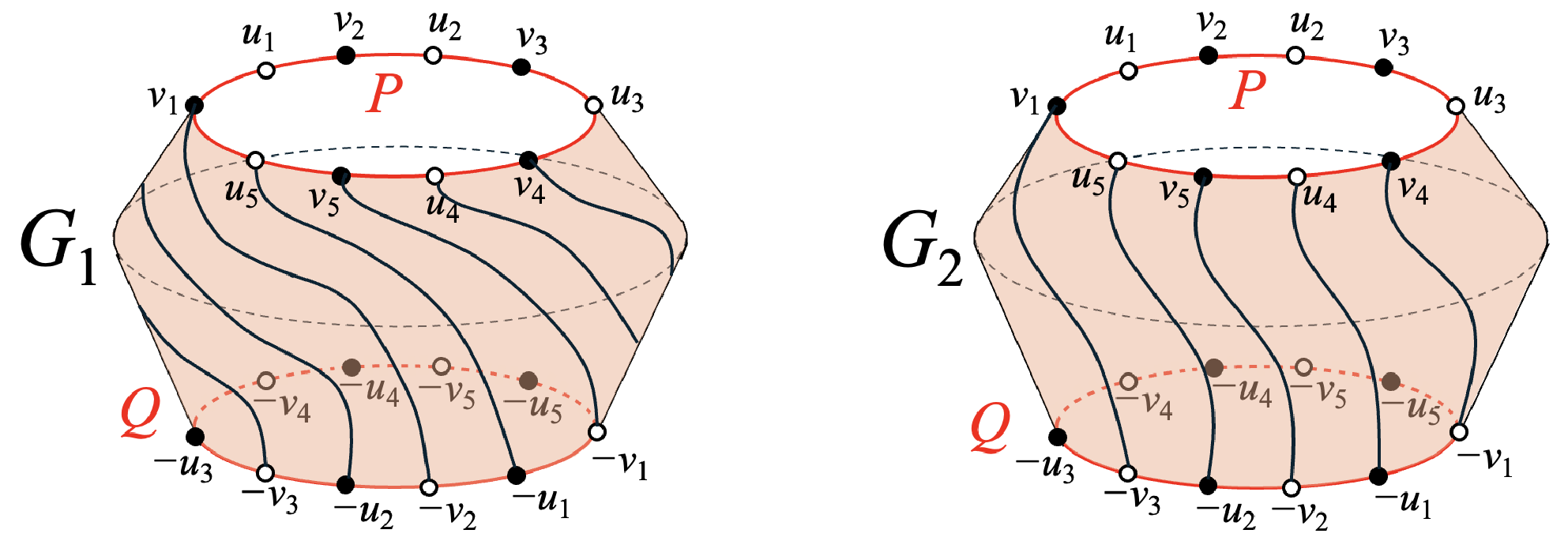}
\end{center}
\caption{Annuli $A_1$ and $A_2$ in the proof of
  Theorem~\ref{t:complete}.}
\label{t1t2}
\end{figure}

Finally, we decompose the interior domain of $B_0$ into cubes. Let
$T_i$ be the solid torus whose boundary is obtained from
$A_{i-1}$ and $A_i$ by pasting them along $P$ and $Q$, for
$i=1,\ldots,k$.  Then we see that the annuli $A_1,\ldots, A_k$
decompose $B_0$ into $T_1,\ldots,T_k$ and a single ball, say $B$,
containing $v_0$, $-v_0$ and $A_k$.  Note that the graph on the
boundary of $T_i$ for each $i$, as well as the graph on $B$, are
bipartite quadrangulations with an even number of faces. Now applying
Theorem~\ref{t:jeff} to each of them, we get a normal quadrangulation
of the interior domain of $B_0$.  Recall that $G$ is the
quadrangulation of the boundary $\partial B_0$.  Since $G$ has the
point-symmetry $\rho$ with respect to $(0,0,0)$, we can identify the
antipodal points on $\partial B_0$ and get a normal quadrangulation of
$\RP 3$.  Let $\tilde{G}$ denote its $1$-skeleton.  Thus, by the
construction, $\tilde{G}$ contains a complete graph on $n+1$ vertices
$\tilde{v}_0, \tilde{v}_1, \ldots, \tilde{v}_n$, where $\tilde{v}_i$
represents the vertex obtained by identifying $v_i$ and $-v_i$ through
the symmetry $\rho$, for $i=0,1,\ldots,n$.
\end{proof}

\section{Alternative proof of Youngs' theorem}
\label{sec:RP2}

In this section,
we give an alternative proof for Youngs' theorem (Theorem~\ref{t-young}),
which states that 
no non-bipartite quadrangulation of the projective plane $\RP2$ 
is $3$-colourable.
By extending the ideas in this proof,
we will prove Theorem~\ref{t:3col} in the next section.

\begin{proof}[Alternative proof of Theorem~\ref{t-young}]
Let $\CC$ be a non-bipartite normal quadrangulation of $\RP2$,
that is,
a quadrangulation of $\RP2$ in the ordinary sense.
Suppose that we are given a proper colouring $c$ of $\CC^{(1)}$, the $1$-skeleton
of $\CC$, with colours $\Setx{1,2,3,4}$. We will show that $c$ 
uses all four colours.

For $t\in\Setx{2,3,4}$, let $E_t$ be the set of all edges of $\CC$
with their ends coloured either by $1$ and $t$, or by the
complementary two colours.
Consider a fixed $t\in\Setx{2,3,4}$. We begin by constructing 
closed curves $H_t$ in $\RP 2$ intersecting precisely those edges of $\CC$ which are
contained in $E_t$.

Let $M$ be the set of midpoints of all edges of $E_t$.  It is easy to
see that the boundary of any face $A$ contains an even number of
points of $M$. If it contains two, then we add a curve joining
them through $A$. If it contains four, then we add two
disjoint non-crossing curves whose ends are these four points. There are two ways
to choose these curves; we choose one arbitrarily.
For an example, see Figure~\ref{E123}.

\begin{figure}[htb]
\begin{center}
\includegraphics[width=100mm]{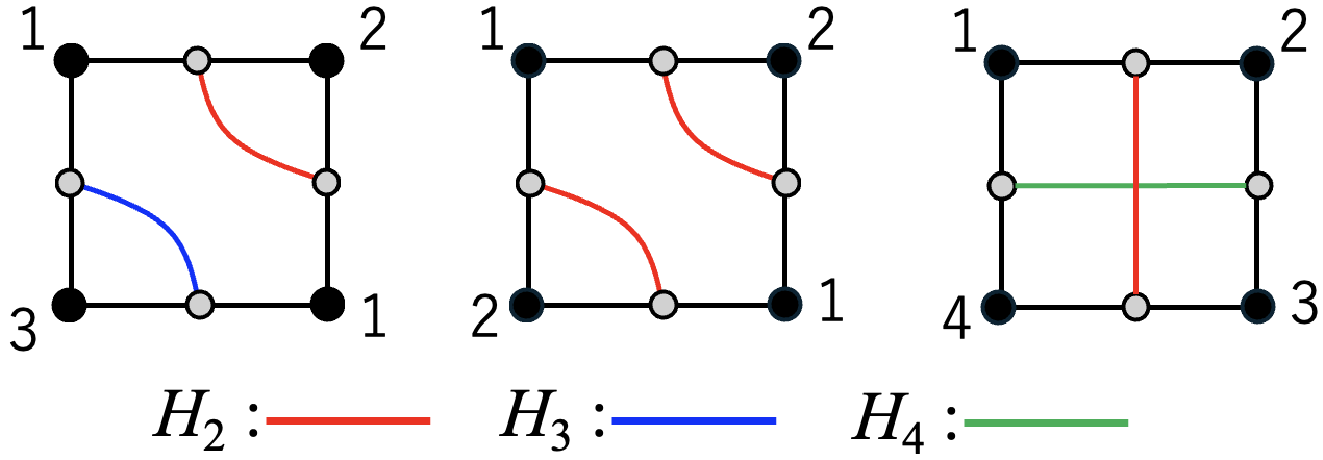}
\end{center}
\caption{Put curves in each face joining midpoints of the edges of $E_t$.}
\label{E123}
\end{figure}

In this way, each point of $M$ meets two of the added curves,
each embedded in one incident face.
Thus, if we define $H_t$
as the union of the curves as $A$ ranges over all faces,
it is a disjoint union of closed curves.
Observe that $\CC^{(1)}$ and $H_t$ intersect transversally,
and their intersection
is a set of points contained in $E_t$.

We have constructed closed curves $H_2$, $H_3$ and $H_4$ of $\RP2$.
Note that if there are two curves inside a face $A$, then they intersect
transversally, the vertices of
$A$ have all four colours, the curves belong to distinct closed
curves among $H_2, H_3$ and $H_4$, and the endpoints of each curve are
at opposite sides of the boundary of $A$.
See Figure~\ref{E123quad}.

\begin{figure}[htb]
\begin{center}
\includegraphics[width=120mm]{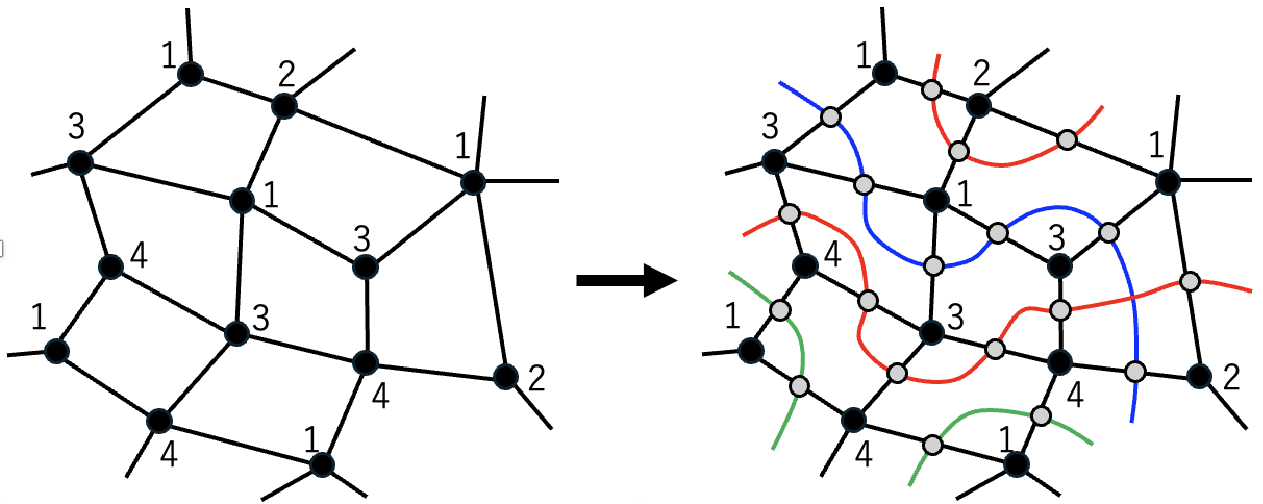}
\end{center}
\caption{Closed curves of $H_2, H_3, H_4$}
\label{E123quad}
\end{figure}

Consider an odd cycle $\gamma$ in $\CC^{(1)}$. We prove that $\gamma$ contains an
odd number of edges from $E_2$. It is sufficient to show that
$\size{E(\gamma) \cap (E_3\cup E_4)}$ is even. The colouring $c$ determines
a homomorphism from $\gamma$ to the complete graph $K_4$ on vertices
$1,2,3,4$. The edges of $E_3\cup E_4$ are precisely the edges mapped
to the set $\Setx{13,24,14,23} \subseteq E(K_4)$, which is an edge-cut
in $K_4$. The image of $\gamma$ under the homomorphism is a closed walk in
$K_4$, which has to use an even number of edges from this
edge-cut. Thus, $\size{E(\gamma) \cap (E_3\cup E_4)}$ is even as
claimed. It follows that $\gamma$ contains an odd number of edges of
$E_2$. By symmetry, the same holds for $E_3$ and $E_4$. Consequently,
$\gamma$ intersects each $H_t$ ($t=2,3,4$) in an odd number of points.

All homotopically trivial cycles of $\CC$ in $\RP2$
are obtained by the symmetric difference 
among boundary cycles of some faces,
and hence they are of even length.  Thus, the odd cycle $\gamma$ must
be nontrivial.  Since $\gamma$ intersects each $H_t$ ($t=2,3,4$)
transversally in an odd number of points, $H_t$ contains a nontrivial
curve on $\RP2$.  Thus, $H_2$ and $H_3$ intersect.  Considering a face
$A$ of $\CC$ such that $H_2$ and $H_3$ intersect inside $A$, the
definition of $H_2$ and $H_3$ implies that the vertices of $A$ must
have all four colours. This concludes the proof.
\end{proof}

\section{Higher-dimensional cases}
\label{sec:RPd}
This section is devoted to generalizing the argument in Section~\ref{sec:RP2} to higher-dimensional cases and proving Theorem~\ref{t:3col}. 
Let $H_k(X)$ (respectively, $H^k(X)$) denote the $k$th singular homology (respectively, cohomology) group of $X$ with coefficients in $\ZZ_2=\ZZ/2\ZZ$. 
For $\alpha\in H^i(X)$ and $\beta\in H^j(X)$, the cup product
$\alpha\smile \beta\in H^{i+j}(X)$ is defined, which makes
$\bigoplus_{k\geq 0} H^k(X)$ into a ring. 
See \cite[Chapter~VI, Section~4]{Bre} for the definition.

\begin{theorem}\label{thm:3col_of_X}
For $d\geq 2$, let $X$ be a connected closed $d$-dimensional manifold satisfying the following conditions.
\begin{enumerate}[label=\textup{(\arabic*)}]
    \item For any non-trivial element $\alpha\in H^{d-2}(X)$, there exists $\beta\in H^2(X)$ such that $\alpha\smile\beta \neq 0$.
    \item For any non-trivial elements $\alpha,\beta\in H^{1}(X)$, it holds that $\alpha\smile\beta \neq 0$.
\end{enumerate}
Then, no non-bipartite normal quadrangulation of $X$ is $3$-colourable.
\end{theorem}

\begin{remark}
Poincar\'{e} duality allows us to rephrase the above conditions as follows.
\begin{enumerate}[label=\textup{(\arabic*)}]
    \item For any non-trivial element $\alpha\in H_{d-2}(X)$, there exists $\beta\in H_{2}(X)$ such that $\alpha\cdot\beta \neq 0$.
    \item For any non-trivial elements $\alpha,\beta\in H_{d-1}(X)$, it holds that $\alpha\cdot\beta \neq 0$.
\end{enumerate}
Here, $\cdot$ denotes the intersection between homology classes.  If
submanifolds $Y$ and $Z$ intersect transversally, then
$[Y]\cdot [Z]=[Y\cap Z]$ holds, where $[Y]$ denotes the fundamental
class of $Y$.  See \cite[Chapter~VI, Section~11]{Bre} for details.
\end{remark}

The $d$-dimensional real projective space $\RP{d}$ satisfies the above conditions since its cohomology ring $H^\ast(\RP{d})$ is isomorphic to the quotient $\ZZ_2[w]/(w^{d+1})$ of a polynomial ring (see \cite[Chapter~VI, Proposition~10.2]{Bre} for instance).
For $p\geq 2$ and integers $q_1,\dots,q_n$ relatively prime to $p$, the lens space $L(p;q_1,\dots,q_n)$ is defined, for example, $L(2;1,\dots,1) = \RP{2n-1}$.
See \cite[Chapter~III, Example~4.7]{Bre} for the precise definition.
If $p$ is even, then $H^\ast(L(p;q_1,\dots,q_n))$ is isomorphic to $H^\ast(\RP{d})$ as rings.
Therefore, we obtain the next result as a consequence of Theorem~\ref{thm:3col_of_X}. 

\begin{corollary}\label{cor:3col_of_RPd}
No non-bipartite normal quadrangulations of $\RP{d}$ and $L(p;q_1,\dots,q_n)$ for $p$ even is $3$-colourable.
\end{corollary}

In contrast, for a closed surface $\Sigma$, the cohomology ring of the closed $3$-manifold $\Sigma\times S^1$ does not satisfy the condition (2), and we can easily find a non-bipartite normal quadrangulation of the $3$-torus $T^2\times S^1$ whose $1$-skeleton is $3$-colourable.

Let $\CC$ be a normal quadrangulation of a closed $d$-manifold $X$.
We write $\CC^{(k)}$ for the $k$-skeleton of $\CC$.

Throughout the rest of this section, we assume that a colouring of $\CC^{(1)}$ with colours $1,\dots,4$ is
given. 
We obtain ``hypersurfaces'' $H_2$, $H_3$, and $H_4$ as in Section~\ref{sec:RP2}.
More precisely, we consider $d\cdot 2^{d-1}$ $(d-1)$-dimensional cubes in $[0,1]^d$ obtained from the cubes $[0,1]^{k-1}\times\{1/2\}\times[0,1]^{d-k}$ ($k=1,\dots,d$) by cutting them along their intersections.
For each edge of $[0,1]^d$, there is a unique small hyperplane intersecting the edge.
We now define $H_t$ to be the union of small hyperplanes $H$ in $\CC$ such that the ends of the edge corresponding to $H$ are colored by $1$ and $t$, where $t\in\{2,3,4\}$.

\begin{remark}
  The proof does not assume that our normal quadrangulations are combinatorial
  cubulations (i.e., that the link of each $k$-cube is a
  piecewise-linear $(d-k-1)$-sphere).  See \cite[Section~2]{Funar}.
\end{remark}

\begin{lemma}\label{lem:skeleton}
Let $\CC$ be a quadrangulation of a closed $d$-manifold $X$ satisfying the condition $(1)$, where $d\geq 2$.
Let $Y$ be a closed $(d-2)$-submanifold with $Y\cap \CC^{(2)}=\emptyset$.
Then, the homology class $[Y]\in H_{d-2}(X)$ is trivial.
\end{lemma}

\begin{proof}
  The assumption $Y\cap \CC^{(2)}=\emptyset$ implies that there exists
  a neighborhood $U$ of $\CC^{(2)}$ disjoint from $Y$.  Assume $[Y]$
  is non-trivial for a contradiction.  Then the condition (1) implies
  that there is $\beta\in H_2(X)$ with $[Y]\cdot \beta\neq 0$.  By
  \cite{Tho54}, we have a submanifold $\Sigma$ representing $\beta$,
  i.e., $[\Sigma]=\beta$.  Note here that for any point $p$ in the
  interior of a $k$-dimensional disk $D^k$, the complement
  $D^k\setminus\{p\}$ deformation retracts to $\partial D^k=S^{k-1}$.
  Thus, by an isotopy, we may assume $\Sigma$ lies in a neighborhood
  of $\CC^{(d-1)}$.  Continuing this process, we deform $\Sigma$ so
  that it is contained in the neighborhood $U$ since $\Sigma$ is of
  dimension $2$.  Hence, $Y\cap \Sigma= \emptyset$, and thus,
  $[Y]\cdot \beta=0$.  This is a contradiction.
\end{proof}

For example, in the case $X=\RP{d}$, the above submanifold $\Sigma$ can be chosen as $\RP{2}$ embedded in a standard way.

\begin{lemma}\label{lem:extension}
For $k\geq 2$, let $Y$ be a \textup{(}possibly disconnected\textup{)} closed $(k-1)$-submanifold of $S^{k}=\partial D^{k+1}$.
Then, $Y$ bounds a properly embedded compact $k$-manifold in $D^{k+1}$, that is, there exists a compact submanifold $W$ of $D^{k+1}$ satisfying $\partial W=Y$ and $W\cap S^{k} = \partial W$. 
\end{lemma}

\begin{proof}
Let $Y=\bigsqcup_{i=1}^n Y_i$, where each $Y_i$ is connected and
$\bigsqcup$ denotes disjoint union.
Let $\widehat{Y}_i$ be one of the two connected components of $S^k\setminus Y_i$.
Note that if $\widehat{Y}_i\cap \widehat{Y}_j\neq \emptyset$, then either $\widehat{Y}_i\subseteq \widehat{Y}_j$ or $\widehat{Y}_i\supseteq \widehat{Y}_j$ holds.
We may assume that if $\widehat{Y}_i\supseteq \widehat{Y}_j$, then $i<j$.
Let $W_i = \widehat{Y}_i\times\{i/n\}\cup Y_i\times[i/n,1]$ in the cone $D^{k+1}=(S^k\times [0,1])/(S^k\times \{0\})$. 
By the construction, $W_i$ are disjoint submanifolds of $D^{k+1}$ and $W=\bigsqcup_{i=1}^n W_i$ is a desired submanifold.
\end{proof}

For each $i\in \{2,3,4\}$, let us construct a closed
$(d-1)$-submanifold $Y_i$ of $X$ from $H_i$. The construction is as
follows.  We first fix a way of resolving self-intersections of
$H_i\cap \CC^{(2)}$ on every $2$-cell by changing the intersecting
curves to disjoint ones, as done in Section~\ref{sec:RP2}. Then, for
each $3$-cell $C^3$, we have a (possibly disconnected) closed
$1$-submanifold $Y^1$ of $\partial C^3\cong S^2$.  By
Lemma~\ref{lem:extension}, $Y^1$ bounds a $2$-manifold $W^2$ in $C^3$.
Then, for each $4$-cell $C^4$, we have a (possibly disconnected)
closed $2$-submanifold $Y^2$ of $\partial C^4\cong S^3$.  Indeed, in
$\partial C^4$, there are exactly two $3$-cells $C^3_1$ and $C^3_2$
containing a $2$-cell $C^2$, and $W^2_1\subset C^3_1$ and
$W^2_2\subset C^3_2$ are glued together along
$W^2_1\cap C^2 = W^2_2\cap C^2$.  We can continue this process and
obtain a closed $(d-1)$-submanifold $Y^{d-1}$ of $X$ and let
$Y_i=Y^{d-1}$.

\begin{lemma}\label{lem:non-trivial}
The homology class $[Y_i]$ is non-trivial in $H_{d-1}(X)\cong \ZZ_2$ for $i\in \{2,3,4\}$.
\end{lemma}

\begin{proof}
Since $\CC^{(1)}$ is non-bipartite, there exists an odd cycle
$\gamma$. 
By the construction of $Y_i$ ($i\in\{2,3,4\}$), the intersection of $\gamma$ and $Y_2\cup Y_3\cup Y_4$ consists of odd number of transverse double points.
Therefore, $[\gamma]\cdot ([Y_2]+[Y_3]+[Y_4])=1$ in $H_0(X)=\ZZ_2$.
On the other hand, by considering a graph homomorphism $\CC^{(1)}\to
K_4$ induced from the $4$-coloring, we have $[\gamma]\cdot
([Y_2]+[Y_3])=0$. It follows that $[\gamma]\cdot [Y_4]=1$. 
A similar
argument applies to $Y_2$ and $Y_3$.
\end{proof}

\begin{proof}[Proof of Theorem~\ref{thm:3col_of_X}]
We prove in fact that for any distinct $i,j\in \{2,3,4\}$, there exists a $2$-cube $C$ such that two curves $H_i\cap C$ and $H_j\cap C$ intersect transversally at a point in $C$.
Suppose, for a contradiction, that there is no such $2$-cube $C$ for some $i,j$.
Then, $(Y_i\cap Y_j)\cap \CC^{(2)}= \emptyset$.
Here we may assume $Y_i$ and $Y_j$ intersect transversally.
By Lemma~\ref{lem:skeleton}, we have $[Y_i\cap Y_j]=0 \in H_{d-2}(\RP{d})$.
On the other hand, Lemma~\ref{lem:non-trivial} and the condition (2) imply $[Y_i]\cdot [Y_j]\neq 0$.
This is a contradiction.

As in the proof of Theorem~\ref{t-young}, the vertices of the cube $C$ have all four colours. 
This proves Theorem~\ref{thm:3col_of_X}.
\end{proof}

\begin{remark}
In the case $X=\RP{d}$ for $d\geq 3$, we further show that there exists a $3$-cube such that all of the four colours appear on its vertices.
Indeed, if there is no such a $3$-cube, then $(Y_2\cap Y_3\cap Y_4)\cap \CC^{(3)}= \emptyset$.
Hence, $[Y_2]\cdot [Y_3]\cdot [Y_4]=0 \in H_{d-3}(\RP{d})$, which contradicts $w^3\neq 0$.
\end{remark}

\section*{Acknowledgments}
The research of the fourth author, Y. Nozaki, was supported by JSPS KAKENHI Grant Numbers JP20K14317, JP23K12974 and JP24H00686.
The research of the fifth author, K. Ozeki, was supported by JSPS KAKENHI Grant Numbers JP22KF0148,	JP22K19773, and JP23K03195.

\end{document}